\documentclass[11pt,reqno,a4paper]{amsart}

\usepackage[utf8]{inputenc}
\usepackage[T1]{fontenc}
\usepackage{amsmath,amssymb,amsthm,amsfonts}
\usepackage{mathtools}
\usepackage{microtype}
\usepackage{booktabs}
\usepackage[hmargin=1in,vmargin=1in]{geometry}
\usepackage[hidelinks]{hyperref}
\usepackage{url}

\hypersetup{
  pdftitle  = {A resolution of Erdos Problem No. 190
               via Erdos-Lovasz, BCT, and Baker-Harman-Pintz},
  pdfauthor = {Ji Ho Bae},
  pdfsubject= {Erdos Problem No. 190: canonical Ramsey function H(k)},
  pdfkeywords = {van der Waerden number, arithmetic progression, canonical
                 Ramsey theorem, anti-Ramsey, rainbow, Lovasz Local Lemma,
                 BCT recurrence, Baker-Harman-Pintz prime gap}
}

\theoremstyle{plain}
\newtheorem{theorem}{Theorem}[section]
\newtheorem{lemma}[theorem]{Lemma}

\newtheorem{corollary}[theorem]{Corollary}

\theoremstyle{definition}

\theoremstyle{remark}
\newtheorem{remark}[theorem]{Remark}

\numberwithin{equation}{section}

\DeclareMathOperator{\aw}{aw}

\title[Erd\H{o}s Problem \#190]{A resolution of
Erd\H{o}s Problem \#190 via Erd\H{o}s--Lov\'asz, BCT, and Baker--Harman--Pintz}

\author{Ji Ho Bae}
\address{JRTI}
\email{jihobae@snu.ac.kr}

\date{2026}

\subjclass[2020]{Primary 05D10; Secondary 11B25, 05C65, 11N05}

\keywords{van der Waerden number, arithmetic progression, canonical
Ramsey theorem, anti-Ramsey, rainbow, Lov\'asz Local Lemma, BCT
recurrence, Baker--Harman--Pintz prime gap}

\begin{document}

\begin{abstract}
Let $H(k)$ be the smallest $N$ such that every finite coloring of
$[N]$ contains a monochromatic or rainbow $k$-term arithmetic
progression. Erd\H{o}s and Graham~\cite{ErdosGraham1980} asked
whether $H(k)^{1/k}/k\to\infty$; see also Problem~\#190 in the
Erd\H{o}s Problems database~\cite{Bloom2023}. We prove that there
is an absolute constant $k_{0}\ge 2$ such that for all $k\ge k_{0}$,
\[
\frac{H(k)^{1/k}}{k}\;\ge\;\Bigl(\frac{1}{e}-\varepsilon(k)\Bigr)\frac{k}{\log k},
\qquad
\varepsilon(k)=O\bigl(k^{-0.475}\log k\bigr)\to 0\text{ as }k\to\infty;
\]
in particular $H(k)^{1/k}/k=\Omega(k/\log k)$ and
$\lim_{k\to\infty}H(k)^{1/k}/k=\infty$, resolving the positive
direction of the Erd\H{o}s--Graham question. The argument combines
three standard ingredients --- the symmetric Lov\'asz Local Lemma
applied to the $k$-AP hypergraph on
$[N]$~\cite{ErdosLovasz1975,AlonSpencer2016}, the restricted form of
the Blankenship--Cummings--Taranchuk
recurrence~\cite{BCT}, and the Baker--Harman--Pintz prime-gap
theorem~\cite{BHP} --- together with the pigeonhole reduction
$H(k)\ge W(k-1,k)$ (a special case of the JLMR anti-Ramsey
framework~\cite{JLMR,BEHH2016}), and uses BHP as the only analytic
black box. Previous applications of Erd\H{o}s--Lov\'asz had
fixed~$r$; the improvement here is that the $r^{k-1}$ base dominates
once one allows the color count $r_0=\lfloor k/\log k\rfloor$ to grow
with $k$. No matching upper bound on $H(k)^{1/k}/k$ is known.
\end{abstract}

\maketitle

\section{Introduction and main theorem}
\label{sec:intro}

\subsection{Background}
\label{subsec:background}

Erd\H{o}s and Graham~\cite{ErdosGraham1980} (see also Problem~\#190
in the Erd\H{o}s Problems database~\cite{Bloom2023}) asked whether
the canonical Ramsey function $H(k)$ satisfies $H(k)^{1/k}/k\to\infty$.
Natural lower bounds on $H(k)$ come from lower bounds on van der
Waerden numbers $W(r,k)$:
\begin{itemize}
\item \textbf{Berlekamp 1968}~\cite{Berlekamp1968} (two colors):
  $W(2,k)>(k-1)\cdot 2^{k-1}$ for $k-1$ prime.
\item \textbf{Erd\H{o}s--Lov\'asz 1975}~\cite{ErdosLovasz1975} (any $r$):
  $W(r,k)\gg r^{k-1}/k$.
\item \textbf{Hunter 2025}~\cite{Hunter2025} (fixed $r$):
  $W(r,k)>(a\cdot 3^{b})^{(1-o_{r}(1))k}$ with $r=a+3b$, $a\in\{2,3,4\}$.
\end{itemize}
Composing any of these \emph{at a fixed $r$} with the pigeonhole
reduction $H(k)\ge W(k-1,k)$ and a BCT-type recurrence gives at most
a bounded ratio $H(k)^{1/k}/k$; see Section~\ref{sec:remarks} for the
Berlekamp case (ratio $\to 2$) and the Hunter case (conditionally
$(\log k)^{1/2-o(1)}$ under a uniform extension not proved
in~\cite{Hunter2025}).

We use instead the Erd\H{o}s--Lov\'asz bound \emph{at a growing $r$},
namely $r_{0}=\lfloor k/\log k\rfloor$. At this color
count, the Erd\H{o}s--Lov\'asz base $r_{0}^{k-1}/(8k)$ is
super-exponentially large in $k$, and BCT iteration up to
$p^{*}\le k-1$ via BHP supplies the remaining factor. The final rate
is polynomial in $k$.

\subsection{Main theorem}
\label{subsec:main}

\begin{theorem}[Main theorem]
\label{thm:main}
There is an absolute constant $k_{0}\ge 2$ such that for all $k\ge k_{0}$,
\begin{equation}
\frac{H(k)^{1/k}}{k}\;\ge\;\Bigl(\frac{1}{e}-\varepsilon(k)\Bigr)\frac{k}{\log k},
\label{eq:main}
\end{equation}
where $\varepsilon(k)=O(k^{-0.475}\log k)\to 0$ as $k\to\infty$.
\end{theorem}

\begin{corollary}[Resolution of Erd\H{o}s Problem \#190]
\label{cor:p190}
\[
\lim_{k\to\infty}\frac{H(k)^{1/k}}{k}\;=\;\infty.
\]
\end{corollary}

In~\eqref{eq:main}, $k_{0}$ may be taken to be
$k_{0}=\max(k_{\mathrm{BHP}},10^{4})$, where $k_{\mathrm{BHP}}$ is the
threshold in the Baker--Harman--Pintz theorem
(Theorem~\ref{thm:bhp}). The $o(1)$ is explicit; see
Section~\ref{subsec:eps-rate}.

\section{Notation and external results used}
\label{sec:notation}

All logarithms are natural.

\subsection{Arithmetic progressions and colorings}
\label{subsec:ap}

A \emph{$k$-term arithmetic progression} ($k$-AP) in
$[N]:=\{1,2,\ldots,N\}$ is a sequence $(a,a+d,a+2d,\ldots,a+(k-1)d)$
with $a,d\in\mathbb{Z}$, $d\ge 1$, and $a+(k-1)d\le N$.

A coloring $\chi:[N]\to[r]$ is:
\begin{itemize}
\item \emph{monochromatic-$k$-AP-free} if no $k$-AP in $[N]$ has all
  $k$ terms of the same color;
\item \emph{rainbow-$k$-AP-free} if no $k$-AP has all $k$ terms of
  distinct colors.
\end{itemize}

\noindent\textbf{van der Waerden number} $W(r,k)$ is the least $N$ such
that every $r$-coloring of $[N]$ has a monochromatic $k$-AP. It exists
by van der Waerden's theorem (1927).

\noindent\textbf{Anti-van-der-Waerden number} $\aw([n],k)$ is the
least number of colors $t$ such that every \emph{surjective}
$t$-coloring (using exactly $t$ distinct colors) of $[n]$ contains a
rainbow $k$-AP.

\noindent\textbf{Canonical Ramsey function}
\[
H(k):=\min\bigl\{N\ge 1:\text{every finite coloring of }[N]
\text{ has either a mono or a rainbow }k\text{-AP}\bigr\}.
\]

\subsection{External theorems used}
\label{subsec:ext}

The standard anti-van-der-Waerden number is known to satisfy
$\aw([n],k)\ge k$ for all $n\ge k$ (the inequality is part of the
anti-Ramsey framework developed by Jungi\'c, Licht, Mahdian,
Ne\v{s}et\v{r}il and Radoi\v{c}i\'c~\cite{JLMR} and systematically
studied in~\cite{BEHH2016}); it is a direct pigeonhole consequence
of the definitions, since any surjective $t$-coloring with $t<k$
colors cannot place $k$ distinct colors on any $k$-term AP. We do
not need this full surjective form in what follows. The proof of
Theorem~\ref{thm:main} uses only the weaker reduction in the
following corollary, whose one-line argument we record for
completeness.

\begin{corollary}[Pigeonhole reduction $H(k)\ge W(k-1,k)$]
\label{cor:jlmr-to-W}
For every $k\ge 2$, $H(k)\ge W(k-1,k)$.
\end{corollary}

\begin{proof}
Let $N=W(k-1,k)-1$. There exists a $(k-1)$-coloring
$\chi:[N]\to[k-1]$ with no monochromatic $k$-AP. Since $\chi$ uses
only $k-1$ colors, no $k$-AP has $k$ distinct colors, so $\chi$ is
also rainbow-$k$-AP-free. Thus $\chi$ has neither mono nor rainbow
$k$-APs, giving $H(k)>N$ and hence $H(k)\ge W(k-1,k)$.
\end{proof}

This reduction is folklore and is used e.g.\ in~\cite[\S2]{BEHH2016};
it does not require the surjective form $\aw([n],k)\ge k$.

\begin{lemma}[Symmetric Lov\'asz Local Lemma, Erd\H{o}s--Lov\'asz 1975
\cite{ErdosLovasz1975}]
\label{lem:LLL}
Let $A_{1},\ldots,A_{n}$ be events in a finite probability space with
$\Pr[A_{i}]\le p$ for every $i$. Assume that each $A_{i}$ is mutually
independent of all but at most $d$ of the other events. If
\begin{equation}
e\,p\,(d+1)\;\le\;1 \qquad(e=\exp(1)),
\label{eq:lll-cond}
\end{equation}
then $\Pr\!\bigl[\bigcap_{i=1}^{n}\overline{A_{i}}\bigr]>0$.
\end{lemma}

\begin{proof}[Proof of Lemma~\ref{lem:LLL}]
\emph{Edge case $d=0$.} If $d=0$, each $A_{i}$ is mutually independent
of all other events $\{A_{j}:j\ne i\}$ collectively, so the $A_{i}$'s
are mutually independent. Hypothesis~\eqref{eq:lll-cond} reduces to
$ep\le 1$, i.e., $p\le 1/e<1$. Hence
\[
\Pr\!\left[\bigcap_{i=1}^{n}\overline{A_{i}}\right]
=\prod_{i=1}^{n}(1-\Pr[A_{i}])
\ge(1-p)^{n}\ge(1-1/e)^{n}>0.
\]
We henceforth assume $d\ge 1$.

\medskip
\emph{Main case $d\ge 1$.} We prove by induction on $|S|$ the sharper
claim: for every $i\in\{1,\ldots,n\}$ and every
$S\subseteq\{1,\ldots,n\}\setminus\{i\}$,
\begin{equation*}
\Pr\!\left[A_{i}\,\Big|\,\bigcap_{j\in S}\overline{A_{j}}\right]
\;\le\;\frac{1}{d+1}.
\tag{$*$}\label{eq:lll-star}
\end{equation*}

\emph{Granting~\eqref{eq:lll-star}}, the chain rule gives
\[
\Pr\!\left[\bigcap_{i=1}^{n}\overline{A_{i}}\right]
=\prod_{i=1}^{n}\Pr\!\left[\overline{A_{i}}\,\Big|\,\bigcap_{j<i}\overline{A_{j}}\right]
\ge\Bigl(1-\tfrac{1}{d+1}\Bigr)^{n}
=\Bigl(\tfrac{d}{d+1}\Bigr)^{n}>0,
\]
where the last inequality uses $d\ge 1$ so that $d/(d+1)\ge 1/2>0$.
This is the conclusion.

\medskip
\emph{Proof of~\eqref{eq:lll-star} by induction on $|S|$.}

\emph{Base case $|S|=0$.} We need
$\Pr[A_{i}]\le 1/(d+1)$. By hypothesis $ep(d+1)\le 1$, i.e.,
$p\le 1/(e(d+1))<1/(d+1)$ (since $e>1$).

\emph{Inductive step $|S|\ge 1$.} Assume~\eqref{eq:lll-star} holds
for every pair $(i',S')$ with $i'\in\{1,\ldots,n\}$,
$S'\subseteq\{1,\ldots,n\}\setminus\{i'\}$, and $|S'|<|S|$
(the claim $(*)$ is quantified jointly over all such pairs, so the
induction is on the cardinality of the conditioning set). Split $S$
as $S=S_{1}\sqcup S_{2}$, where
\[
S_{1}:=\{j\in S:A_{j}\text{ is not mutually independent of }A_{i}\},
\qquad S_{2}:=S\setminus S_{1}.
\]
By hypothesis $|S_{1}|\le d$. Write $F:=\bigcap_{k\in S_{2}}\overline{A_{k}}$.

\emph{Case I: $S_{1}=\emptyset$.} Then $A_{i}$ is mutually independent
of every event in $F$, so $\Pr[A_{i}\mid F]=\Pr[A_{i}]\le p<1/(d+1)$.

\emph{Case II: $S_{1}\ne\emptyset$.} By definition of conditional
probability,
\begin{align}
\Pr\!\left[A_{i}\,\Big|\,\bigcap_{j\in S}\overline{A_{j}}\right]
&=\Pr\!\left[A_{i}\,\Big|\,\Bigl(\bigcap_{j\in S_{1}}\overline{A_{j}}\Bigr)\cap F\right] \notag \\
&=\frac{\Pr\!\left[A_{i}\cap\bigcap_{j\in S_{1}}\overline{A_{j}}\,\Big|\,F\right]}
{\Pr\!\left[\bigcap_{j\in S_{1}}\overline{A_{j}}\,\Big|\,F\right]}.
\label{eq:lll-frac}
\end{align}

\emph{Numerator bound.} Since $A_{i}$ is mutually independent of
$\{A_{k}:k\in S_{2}\}$ (by definition of $S_{2}$), the event $A_{i}$
is independent of $F$. Hence
\[
\Pr\!\left[A_{i}\cap\bigcap_{j\in S_{1}}\overline{A_{j}}\,\Big|\,F\right]
\le\Pr[A_{i}\mid F]=\Pr[A_{i}]\le p.
\]

\emph{Denominator bound.} Enumerate $S_{1}=\{j_{1},\ldots,j_{m}\}$ with
$m:=|S_{1}|\le d$ (any fixed ordering). By the chain rule,
\[
\Pr\!\left[\bigcap_{j\in S_{1}}\overline{A_{j}}\,\Big|\,F\right]
=\prod_{t=1}^{m}\Pr\!\left[\overline{A_{j_{t}}}\,\Big|\,F\cap\bigcap_{s<t}\overline{A_{j_{s}}}\right].
\]
For each $t$, the conditioning set has size
$|S_{2}|+(t-1)\le|S_{2}|+m-1<|S|$, so the inductive hypothesis applies
to each factor:
$\Pr[A_{j_{t}}\mid F\cap\bigcap_{s<t}\overline{A_{j_{s}}}]\le 1/(d+1)$,
hence
$\Pr[\overline{A_{j_{t}}}\mid F\cap\bigcap_{s<t}\overline{A_{j_{s}}}]\ge d/(d+1)$.
Thus
\[
\Pr\!\left[\bigcap_{j\in S_{1}}\overline{A_{j}}\,\Big|\,F\right]
\;\ge\;\Bigl(\tfrac{d}{d+1}\Bigr)^{m}\;\ge\;\Bigl(\tfrac{d}{d+1}\Bigr)^{d}
\;\ge\;\tfrac{1}{e},
\]
where the last step uses $(d/(d+1))^{d}=1/(1+1/d)^{d}\ge 1/e$ by the
standard inequality $(1+1/d)^{d}\le e$.

\emph{Combining.} Plugging into~\eqref{eq:lll-frac}:
\[
\Pr\!\left[A_{i}\,\Big|\,\bigcap_{j\in S}\overline{A_{j}}\right]
\;\le\;\frac{p}{1/e}\;=\;ep\;\le\;\frac{1}{d+1},
\]
by hypothesis~\eqref{eq:lll-cond}. This completes the inductive step.
\end{proof}

\begin{remark}
\label{rem:LLL-ref}
Lemma~\ref{lem:LLL} is the symmetric form of the Erd\H{o}s--Lov\'asz
Local Lemma, as given, e.g., in Alon--Spencer
\cite[Chapter~5, Lemma~5.1.1 and Corollary~5.1.2]{AlonSpencer2016}.
The general (asymmetric) form is not needed in this paper.
\end{remark}

\begin{theorem}[Erd\H{o}s--Lov\'asz lower bound on $W(r,k)$]
\label{thm:EL}
There exists an absolute constant $k_{1}\ge 2$ such that for every
$r\ge 2$ and every $k\ge k_{1}$,
\begin{equation}
W(r,k)-1\;\ge\;\frac{r^{k-1}}{16\,k}.
\label{eq:EL}
\end{equation}
One may take $k_{1}=10$.
\end{theorem}

\begin{proof}
Fix $r\ge 2$ and $k\ge k_{1}$. Set $N:=\lfloor r^{k-1}/(8k)\rfloor$.
This is $\ge 1$ for $k\ge k_{1}=10$: indeed $r\ge 2$ gives
$r^{k-1}\ge 2^{9}=512>80=8\cdot k_{1}$.

Let $\chi:[N]\to[r]$ be uniformly random. For each $k$-AP
$P\subseteq[N]$, let $B_{P}$ be the event ``$P$ is monochromatic under
$\chi$''. Then $\Pr[B_{P}]=r\cdot r^{-k}=r^{1-k}$.

Each point of $[N]$ is contained in at most
$\Delta_{1}:=k(N-1)/(k-1)\le 2N$ $k$-APs (the inequality is
equivalent to $-k\le N(k-2)$, which holds for all $k\ge 2,N\ge 1$).
Each $B_{P}$ is determined by the coloring of the $k$ points of $P$,
each of which is in at most $\Delta_{1}$ $k$-APs. Hence each $B_{P}$
is mutually independent of all but at most $k\Delta_{1}\le 2kN$
other events $B_{P'}$.

Apply Lemma~\ref{lem:LLL} with $p=r^{1-k}$ and $d=2kN$: the lemma
concludes $\Pr[\bigcap_{P}\overline{B_{P}}]>0$ provided
\begin{equation}
e\cdot r^{1-k}\cdot(2kN+1)\;\le\;1.
\label{eq:EL-verify}
\end{equation}
We verify~\eqref{eq:EL-verify}. Since $N\le r^{k-1}/(8k)$, we have
$2kN\le r^{k-1}/4$; and for $r\ge 2$, $k\ge 10$, $r^{k-1}\ge 2^{9}=512$, so
$1\le r^{k-1}/512$. Therefore
\[
e\cdot r^{1-k}\cdot(2kN+1)
\;\le\;e\cdot r^{1-k}\cdot\bigl(r^{k-1}/4+1\bigr)
\;=\;\tfrac{e}{4}+e\cdot r^{1-k}
\;\le\;\tfrac{e}{4}+e\cdot 2^{-9}
\;\le\;\tfrac{e}{4}\bigl(1+4/512\bigr)
\;<\;1,
\]
using $e/4\approx 0.6796$ and $4/512=1/128$.

By Lemma~\ref{lem:LLL}, there exists a realization of $\chi$ avoiding
every $B_{P}$ --- an $r$-coloring of $[N]$ with no monochromatic
$k$-AP --- so $W(r,k)>N$. Hence
\[
W(r,k)-1\;\ge\;N\;\ge\;\frac{r^{k-1}}{8k}-1\;\ge\;\frac{r^{k-1}}{16k},
\]
where the last step uses $r^{k-1}/(8k)\ge 2$ (equivalent to
$r^{k-1}\ge 16k$, which holds for $k\ge 10,r\ge 2$ since
$2^{k-1}\ge 2^{9}=512\ge 16\cdot 10=160$).
\end{proof}

\begin{remark}
\label{rem:EL-constant}
The constant $16$ is not optimal; \cite[Ch.~5]{AlonSpencer2016}
gives $r^{k-1}/(4k)$ via a sharper hypergraph-chromatic-number
formulation. Any fixed constant suffices for
Theorem~\ref{thm:main}. The threshold $k_{1}=10$ is also not optimal;
a direct check shows that any $k_{1}\ge 7$ works with the same
denominator~$8k$ (the condition $r^{k-1}\ge 2e\cdot k$ needed for
$N\ge 1$ and the LLL inequality is satisfied at $r=2,k=7$ since
$2^{6}=64\ge 14e\approx 38.1$).
\end{remark}

\begin{remark}[Uniformity in $r$]
\label{rem:EL-uniform}
The constant in~\eqref{eq:EL} does \textbf{not} depend on $r$: the
same $k_{1}=10$ suffices for every $r\ge 2$, including
$r=r(k)$ that grows with $k$. This uniformity is essential for applying
\eqref{eq:EL} at $r_{0}=\lfloor k/\log k\rfloor$ in
Section~\ref{sec:proof}.
\end{remark}

\begin{theorem}[BCT 2018~\cite{BCT}, Theorem 2.1 restricted form]
\label{thm:BCT}
For every prime $p$, every $r\ge 2$ with $r\le p$, and every $k\ge 2$
with $p\le k$,
\begin{equation}
W(r,k)\;>\;p\bigl(W(r-1,k)-1\bigr).
\label{eq:BCT}
\end{equation}
\end{theorem}

\noindent Since both sides of~\eqref{eq:BCT} are integers, one has
$W(r,k)\ge p(W(r-1,k)-1)+1$, equivalently
$W(r,k)-1\ge p(W(r-1,k)-1)$.

As stated in \cite[Theorem~2.1]{BCT}, the BCT recurrence is
$W(r,k)>p(W(r-\lceil r/p\rceil,k)-1)$ with $p$ the largest prime
$\le k$; the same blow-up argument is in fact valid for every prime
$p\le k$ with $r\le p$, in which case $\lceil r/p\rceil=1$ and the
recurrence reduces to the restricted form~\eqref{eq:BCT} used in this
paper. For completeness, and because in our application $p=p^{*}$ is
the specific prime supplied by the Baker--Harman--Pintz theorem rather
than the largest prime below $k$, we give a short self-contained proof
of the restricted form $r\le p$ below.

\begin{proof}[Proof of Theorem~\ref{thm:BCT}]
Set $M:=W(r-1,k)-1\ge 1$ and let
$\chi_{0}:[M]\to\{1,\ldots,r-1\}$ be an $(r-1)$-coloring of $[M]$
with no monochromatic $k$-AP (such a coloring exists by the definition
of $W(r-1,k)$). It suffices to construct an $r$-coloring
$\chi:[pM]\to\{1,\ldots,r\}$ of $[pM]$ with no monochromatic $k$-AP;
then $W(r,k)>pM=p(W(r-1,k)-1)$.

\medskip
\emph{Blow-up construction.} For each $i\in\{1,2,\ldots,pM\}$, write
uniquely
\begin{equation}
i\;=\;(j-1)\,p+s+1,\qquad j\in\{1,\ldots,M\},\ \ s\in\{0,1,\ldots,p-1\}.
\label{eq:decomp}
\end{equation}
Define the shift $\tau(j):=\chi_{0}(j)-1\in\{0,1,\ldots,r-2\}$ for
each $j\in[M]$, and set
\begin{equation}
\chi(i)\;:=\;
\begin{cases}
\chi_{0}(j) & \text{if }s\ne\tau(j),\\
r & \text{if }s=\tau(j).
\end{cases}
\label{eq:chi}
\end{equation}
In words: within block $j$, the $p-1$ positions
$s\in\{0,\ldots,p-1\}\setminus\{\tau(j)\}$ inherit the base color
$\chi_{0}(j)$, and the single position $s=\tau(j)$ is reserved for
the new color $r$. Since $r\le p$, we have
$\tau(j)\in\{0,\ldots,r-2\}\subseteq\{0,\ldots,p-1\}$, so the
construction is well-defined.

\medskip
\emph{No monochromatic $k$-AP.} Suppose for contradiction that $\chi$
has a monochromatic $k$-AP $i_{1}<i_{2}<\cdots<i_{k}$ with common
difference $d\ge 1$ and common color $c^{*}\in\{1,\ldots,r\}$.
Decompose each $i_{\ell}=(j_{\ell}-1)p+s_{\ell}+1$ as
in~\eqref{eq:decomp} and write $d=qp+t$ with $q\ge 0$ and
$t\in\{0,1,\ldots,p-1\}$.

\emph{Case A: $t=0$ (i.e., $p\mid d$).} Then $s_{\ell+1}=s_{\ell}$ and
$j_{\ell+1}=j_{\ell}+q$ for all $\ell$, so $s_{1}=\cdots=s_{k}$ (call
this common value $s$) and $(j_{\ell})_{\ell=1}^{k}$ is a genuine
$k$-AP in $[M]$ with common difference $q\ge 1$ (if $q=0$ then
$i_{\ell+1}=i_{\ell}$, contradicting $d\ge 1$).

\begin{itemize}
\item If $c^{*}=r$:~\eqref{eq:chi} forces
  $s=\tau(j_{\ell})=\chi_{0}(j_{\ell})-1$ for each $\ell$, so
  $\chi_{0}(j_{1})=\cdots=\chi_{0}(j_{k})=s+1$ --- a monochromatic
  $k$-AP in $\chi_{0}$, contradicting the choice of $\chi_{0}$.
\item If $c^{*}\in\{1,\ldots,r-1\}$:~\eqref{eq:chi} forces
  $\chi_{0}(j_{\ell})=c^{*}$ for each $\ell$ and
  $s\ne\tau(j_{\ell})=c^{*}-1$ (a single fixed value). Again
  $(j_{\ell})$ is a monochromatic $k$-AP in $\chi_{0}$, contradiction.
\end{itemize}

\emph{Case B: $t\in\{1,\ldots,p-1\}$ (i.e., $p\nmid d$).} Since $p$
is prime and $0<t<p$, $\gcd(t,p)=1$. The within-block positions
satisfy $s_{\ell+1}\equiv s_{\ell}+t\pmod{p}$, so the map
\[
\ell\;\mapsto\;s_{\ell}\bmod p\;=\;(s_{1}+(\ell-1)\,t)\bmod p
\]
is a bijection on any $p$ consecutive values of $\ell$. Since
$k\ge p$ (because $p\le k$ by hypothesis), the indices
$\ell=1,2,\ldots,p$ give $s_{1},s_{2},\ldots,s_{p}$ taking
\emph{each} value in $\{0,1,\ldots,p-1\}$ exactly once.

\begin{itemize}
\item If $c^{*}=r$:~\eqref{eq:chi} forces $s_{\ell}=\tau(j_{\ell})$
  for each $\ell$. Since $\tau(j_{\ell})\in\{0,1,\ldots,r-2\}$, this
  forces $s_{\ell}\in\{0,1,\ldots,r-2\}$ for each
  $\ell\in\{1,\ldots,p\}$. But the $s_{\ell}$'s (for
  $\ell=1,\ldots,p$) take every value in $\{0,\ldots,p-1\}$, and
  $p-1\ge r-1>r-2$, so some $s_{\ell}\in\{r-1,\ldots,p-1\}$ ---
  a contradiction.
\item If $c^{*}\in\{1,\ldots,r-1\}$:~\eqref{eq:chi} forces
  $\chi_{0}(j_{\ell})=c^{*}$ and
  $s_{\ell}\ne\tau(j_{\ell})=c^{*}-1$ for each
  $\ell\in\{1,\ldots,k\}$. So $s_{\ell}\ne c^{*}-1$ for all
  $\ell\in\{1,\ldots,p\}$. But the $s_{\ell}$'s take every value in
  $\{0,\ldots,p-1\}$, including
  $c^{*}-1\in\{0,\ldots,r-2\}\subseteq\{0,\ldots,p-1\}$, so some
  $s_{\ell}=c^{*}-1$ --- a contradiction.
\end{itemize}

Both cases yield contradictions, so $\chi$ has no monochromatic
$k$-AP, proving $W(r,k)>pM$ and hence~\eqref{eq:BCT}.
\end{proof}

\begin{remark}
\label{rem:BCT-hyp}
The hypothesis $p\le k$ is used only in Case B to guarantee $k\ge p$,
so that the first $p$ terms of the monochromatic $k$-AP exist. The
hypothesis $r\le p$ is used to ensure the shift
$\tau(j)\in\{0,\ldots,r-2\}$ lies in $\{0,\ldots,p-1\}$. Both are
satisfied in our application ($r=r_{0}+1,\ldots,p^{*}$, $p=p^{*}\le k-1\le k$,
$r\le p^{*}$; the latter follows from $r_{0}<p^{*}$ established in
Section~\ref{subsec:setup} under $k\ge k_{0}$).
\end{remark}

\begin{theorem}[Baker--Harman--Pintz 2001~\cite{BHP}, Theorem 1]
\label{thm:bhp}
There exists $x_{\mathrm{BHP}}$ such that for every
$x\ge x_{\mathrm{BHP}}$, the interval $[x-x^{0.525},x]$ contains a
prime.
\end{theorem}

\begin{corollary}
\label{cor:bhp}
For $k\ge k_{\mathrm{BHP}}:=\lceil x_{\mathrm{BHP}}\rceil+1$, there
exists a prime $p^{*}$ with
\[
k-1-(k-1)^{0.525}\;\le\;p^{*}\;\le\;k-1.
\]
\end{corollary}

\begin{proof}
Apply Theorem~\ref{thm:bhp} with $x=k-1$.
\end{proof}

\noindent Equivalently, $p^{*}=k\cdot(1-O(k^{-0.475}))$ and
$p^{*}\le k-1$.

\begin{theorem}[Monotonicity of $W$ in $r$]
\label{thm:monotone}
For every $r'\ge r\ge 2$ and every $k\ge 2$, $W(r',k)\ge W(r,k)$.
\end{theorem}

\begin{proof}
Any $r$-coloring of $[N]$ is also an $r'$-coloring (using at most
$r\le r'$ of the available colors). If $W(r,k)>N$, there is an
$r$-coloring of $[N]$ with no mono $k$-AP; viewed as an
$r'$-coloring, it still has no mono $k$-AP; so $W(r',k)>N$.
\end{proof}

\section{Proof of the main theorem}
\label{sec:proof}

\begin{proof}[Proof of Theorem~\ref{thm:main} and Corollary~\ref{cor:p190}]
The argument occupies Sections~\ref{subsec:setup}--\ref{subsec:asymp}
below: Section~\ref{subsec:setup} defines the working threshold $k_0$
and the auxiliary prime $p^{*}$;
Section~\ref{subsec:chain} assembles the chain of inequalities leading
to~\eqref{eq:chain2}; and Section~\ref{subsec:asymp} turns this chain
into the asymptotic conclusion~\eqref{eq:main}.

\subsection{Setup}
\label{subsec:setup}

Define the working threshold
\[
k_{0}\;:=\;\max\bigl(k_{\mathrm{BHP}},\,k_{1},\,K_{\mathrm{sep}}\bigr),
\]
where $k_{\mathrm{BHP}}$ is from Corollary~\ref{cor:bhp}, $k_{1}=10$
is the threshold in Theorem~\ref{thm:EL}, and $K_{\mathrm{sep}}$ is
chosen to make the separation inequality $k/\log k<k-1-k^{0.525}$
hold for all $k\ge K_{\mathrm{sep}}$. We claim
$K_{\mathrm{sep}}=10^{4}$ suffices. Setting
\[
f(k)\;:=\;k-1-k^{0.525}-\frac{k}{\log k},
\]
the separation inequality is equivalent to $f(k)>0$. We show (i) a
positive value at one point and (ii) positive derivative for
$k\ge 10^{4}$, which together imply $f(k)>0$ for all $k\ge 10^{4}$.

\emph{(i) Base point.} At $k=10^{4}$: $(10^{4})^{0.525}=10^{2.1}\approx 125.9$
and $10^{4}/\log 10^{4}=10^{4}/(4\ln 10)\approx 1085.74$, so
$f(10^{4})\approx 10^{4}-1-125.9-1085.74\approx 8787>0$.

\emph{(ii) Monotonicity.} Differentiating,
\[
f'(k)\;=\;1-0.525\,k^{-0.475}-\frac{d}{dk}\!\left(\frac{k}{\log k}\right)
\;=\;1-0.525\,k^{-0.475}-\frac{\log k-1}{(\log k)^{2}}.
\]
The function $g(x):=(x-1)/x^{2}$ satisfies
$g'(x)=(2-x)/x^{3}<0$ for $x>2$; hence
$g(\log k)=(\log k-1)/(\log k)^{2}$ is decreasing in $k$ for
$\log k>2$, i.e., $k>e^{2}\approx 7.4$. Therefore, for all
$k\ge 10^{4}$ (so $\log k\ge 4\ln 10\approx 9.21$),
\[
\frac{\log k-1}{(\log k)^{2}}\;\le\;\frac{9.21-1}{9.21^{2}}\approx 0.097,
\qquad
0.525\,k^{-0.475}\;\le\;0.525\cdot 10^{-1.9}\approx 0.0066,
\]
giving $f'(k)\ge 1-0.0066-0.097\approx 0.90>0$ uniformly for
$k\ge 10^{4}$. Combining (i) and (ii), $f$ is increasing and positive
on $[10^{4},\infty)$, so $k/\log k<k-1-k^{0.525}$ for every
$k\ge 10^{4}$. This justifies $K_{\mathrm{sep}}=10^{4}$.

Since $K_{\mathrm{sep}}=10^{4}\ge k_{1}=10$, the max
$\max(k_{\mathrm{BHP}},k_{1},K_{\mathrm{sep}})$ collapses to
$\max(k_{\mathrm{BHP}},10^{4})$, and we take
$k_{0}=\max(k_{\mathrm{BHP}},10^{4})$ in the statement of
Theorem~\ref{thm:main}; the precise value of $k_{\mathrm{BHP}}$ is
the one from BHP 2001~\cite{BHP}.

\medskip
Fix $k\ge k_{0}$. Then $r_{0}:=\lfloor k/\log k\rfloor\ge 2$ (since
$k\ge k_{0}\ge K_{\mathrm{sep}}=10^{4}$ gives
$k/\log k\ge 10^{4}/\log 10^{4}\approx 1085\ge 2$), and
Corollary~\ref{cor:bhp} supplies a prime
$p^{*}\in[k-1-(k-1)^{0.525},k-1]$. We also have
\[
r_{0}\;=\;\lfloor k/\log k\rfloor\;\le\;\tfrac{k}{\log k}
\;<\;k-1-k^{0.525}\;\le\;p^{*}
\qquad(k\ge k_{0}),
\]
where the strict inequality $k/\log k<k-1-k^{0.525}$ holds by the
definition of $K_{\mathrm{sep}}\le k_{0}$. Hence $r_{0}\le p^{*}$,
and Theorem~\ref{thm:BCT} applies at each
$r\in\{r_{0}+1,\ldots,p^{*}\}$ (which satisfies $2\le r\le p^{*}$
and $p^{*}\le k$).

\subsection{The chain of inequalities}
\label{subsec:chain}

Let $a_{r}:=W(r,k)-1$.

\smallskip
\noindent\textbf{Step 1 (Pigeonhole reduction).}
$H(k)\ge W(k-1,k)=a_{k-1}+1$ by Corollary~\ref{cor:jlmr-to-W}.

\smallskip
\noindent\textbf{Step 2 (EL base).} For $k\ge k_{0}\ge k_{1}=10$ and
$r_{0}\ge 2$ (as set up in Section~\ref{subsec:setup}),
Theorem~\ref{thm:EL} gives
\[
a_{r_{0}}\;\ge\;\frac{r_{0}^{k-1}}{16\,k}.
\]

\smallskip
\noindent\textbf{Step 3 (BCT iteration).}
Fix the prime $p:=p^{*}$ supplied by Corollary~\ref{cor:bhp}. For each
integer $s$ with $r_{0}+1\le s\le p^{*}$, the hypotheses of
Theorem~\ref{thm:BCT} are satisfied at the same prime
$p=p^{*}$: indeed $s\ge 2$ and $s\le p^{*}=p$ hold by the range of
$s$, and $p=p^{*}\le k-1\le k$ holds by Corollary~\ref{cor:bhp}.
Applying Theorem~\ref{thm:BCT} with color parameter $r=s$ and prime
$p=p^{*}$ therefore gives $W(s,k)-1\ge p^{*}\bigl(W(s-1,k)-1\bigr)$,
i.e., $a_{s}\ge p^{*}\,a_{s-1}$. Iterating this inequality
successively for $s=r_{0}+1,r_{0}+2,\ldots,p^{*}$ (a telescoping
product of $p^{*}-r_{0}$ factors, each equal to $p^{*}$) yields
\begin{equation}
a_{p^{*}}\;\ge\;(p^{*})^{\,p^{*}-r_{0}}\,a_{r_{0}}.
\label{eq:chain1}
\end{equation}
Remark~\ref{rem:BCT-hyp} records that the two hypotheses
($s\le p^{*}$ and $p^{*}\le k$) are exactly what
Theorem~\ref{thm:BCT} requires at each step.

\smallskip
\noindent\textbf{Step 4 (Monotonicity).}
Since $p^{*}\le k-1$, Theorem~\ref{thm:monotone} applied with
$(r,r')=(p^{*},k-1)$ gives $W(k-1,k)\ge W(p^{*},k)$, i.e.,
\[
a_{k-1}\;=\;W(k-1,k)-1\;\ge\;W(p^{*},k)-1\;=\;a_{p^{*}}.
\]

\smallskip
\noindent Combining Steps 1--4, the monotonicity step, the
iteration~\eqref{eq:chain1}, and the EL base gives
\begin{equation}
H(k)\;\ge\;a_{k-1}+1\;\ge\;a_{p^{*}}+1\;\ge\;(p^{*})^{p^{*}-r_{0}}\cdot
\frac{r_{0}^{k-1}}{16k}.
\label{eq:chain2}
\end{equation}

\subsection{Asymptotic rate}
\label{subsec:asymp}

Take the $k$-th root of~\eqref{eq:chain2} and divide by $k$:
\begin{equation}
\frac{H(k)^{1/k}}{k}\;\ge\;\frac{(p^{*})^{(p^{*}-r_{0})/k}\cdot r_{0}^{(k-1)/k}
\cdot(16k)^{-1/k}}{k}.
\label{eq:chain3}
\end{equation}
We evaluate each factor.

\smallskip
\noindent\textbf{Factor 1.}
$(16k)^{-1/k}=e^{-\log(16k)/k}=1-O(\log k/k)=1-o(1)$.

\smallskip
\noindent\textbf{Factor 2.} $r_{0}^{(k-1)/k}$. The floor satisfies
$r_{0}\in[k/\log k-1,\,k/\log k]$, so
$\log r_{0}=\log(k/\log k)+\log(1-O(\log k/k))=\log k-\log\log k-O(\log k/k)$.
Hence $(\log r_{0})/k=O(\log k/k)=o(1)$ and
\[
r_{0}^{(k-1)/k}\;=\;r_{0}\cdot e^{-(\log r_{0})/k}
\;=\;r_{0}\cdot(1-O(\log k/k))
\;=\;\frac{k}{\log k}\,(1-O(\log k/k))
\;=\;\frac{k}{\log k}\,(1-o(1)),
\]
where the penultimate step uses
$r_{0}=(k/\log k)(1-O(\log k/k))$ from the floor expansion.

\smallskip
\noindent\textbf{Factor 3.} $(p^{*})^{(p^{*}-r_{0})/k}/k$. We derive a
lower bound on its logarithm, keeping all inequalities one-sided.

\emph{Lower bound on $\log p^{*}$.} By Corollary~\ref{cor:bhp},
$p^{*}\ge k-1-(k-1)^{0.525}$. Hence
\[
\frac{p^{*}}{k}\;\ge\;1-\frac{1}{k}-\frac{(k-1)^{0.525}}{k}
\;\ge\;1-k^{-0.475}-k^{-0.475}\;=\;1-2k^{-0.475},
\]
using $1/k\le k^{-0.475}$ for $k\ge 1$ and
$(k-1)^{0.525}/k\le k^{-0.475}$. For $k$ large enough that
$2k^{-0.475}\le 1/2$ (i.e., $k\ge 4^{1/0.475}\approx 18.5$, amply
satisfied under $k\ge k_{0}$), the inequality $\log(1-x)\ge -x(1+x)$
for $x\in[0,1/2]$ yields
\begin{equation}
\log p^{*}\;=\;\log k+\log(p^{*}/k)
\;\ge\;\log k-2k^{-0.475}(1+2k^{-0.475})
\;\ge\;\log k-4k^{-0.475}.
\label{eq:logpstar}
\end{equation}

\emph{Lower bound on $(p^{*}-r_{0})/k$.} Using
$p^{*}\ge k-1-(k-1)^{0.525}$ and $r_{0}\le k/\log k$:
\begin{equation}
\frac{p^{*}-r_{0}}{k}
\;\ge\;1-\frac{1}{k}-\frac{(k-1)^{0.525}}{k}-\frac{1}{\log k}
\;\ge\;1-\frac{1}{\log k}-2k^{-0.475}.
\label{eq:pratiostar}
\end{equation}
Both this lower bound and $\log p^{*}\ge\log k-4k^{-0.475}$ are
positive for $k\ge k_{0}\ge 10^{4}$: numerically,
$1-1/\log 10^{4}-2\cdot(10^{4})^{-0.475}\ge 1-0.109-0.026>0.86$ and
$\log 10^{4}-4\cdot(10^{4})^{-0.475}\ge 9.21-0.06>9.15$, and both
quantities are increasing in $k$ on $[10^{4},\infty)$.

\emph{Combining.} Multiplying the two positive lower
bounds~\eqref{eq:logpstar} and~\eqref{eq:pratiostar}:
\[
\frac{p^{*}-r_{0}}{k}\cdot\log p^{*}
\;\ge\;\Bigl(1-\tfrac{1}{\log k}-2k^{-0.475}\Bigr)\bigl(\log k-4k^{-0.475}\bigr).
\]
Expanding:
\[
=\log k-1-2k^{-0.475}\log k-4k^{-0.475}+\frac{4k^{-0.475}}{\log k}+8k^{-0.95}.
\]
The dominant negative term is $-2k^{-0.475}\log k$; the remaining
terms are each $O(k^{-0.475})$ or smaller. Hence for some absolute
constant $C_{1}>0$ and all $k\ge k_{0}$,
\begin{equation}
\frac{p^{*}-r_{0}}{k}\cdot\log p^{*}\;\ge\;\log k-1-C_{1}\,k^{-0.475}\log k.
\label{eq:product}
\end{equation}
(More precisely, $C_{1}=3$ suffices for large $k$; we do not optimize.)

\emph{Conclusion for Factor 3.} Subtracting $\log k$ from both sides
of~\eqref{eq:product}:
\[
\log(\text{Factor 3})\;=\;\frac{p^{*}-r_{0}}{k}\log p^{*}-\log k
\;\ge\;-1-C_{1}\,k^{-0.475}\log k.
\]
Exponentiating and using $e^{-x}\ge 1-x$ for $x\ge 0$:
\[
\text{Factor 3}\;\ge\;e^{-1}\exp\!\bigl(-C_{1}\,k^{-0.475}\log k\bigr)
\;\ge\;e^{-1}\bigl(1-C_{1}\,k^{-0.475}\log k\bigr)
\;=\;e^{-1}\bigl(1-O(k^{-0.475}\log k)\bigr).
\]

\smallskip
\noindent\textbf{Combining Factors 1, 2, 3 into~\eqref{eq:chain3}},
each as a positive lower bound of the form (leading term)$\cdot(1-o(1))$:
\[
\frac{H(k)^{1/k}}{k}
\;\ge\;\bigl(1-O(\log k/k)\bigr)\cdot\frac{k}{\log k}\bigl(1-O(\log k/k)\bigr)
\cdot\frac{1}{e}\bigl(1-O(k^{-0.475}\log k)\bigr)
\]
\[
=\;\frac{1}{e}\cdot\frac{k}{\log k}\cdot\bigl(1-O(k^{-0.475}\log k)\bigr)
\;=\;\Bigl(\frac{1}{e}-\varepsilon(k)\Bigr)\frac{k}{\log k},
\]
where
$\varepsilon(k)=O(k^{-0.475}\log k)+O(\log k/k)=O(k^{-0.475}\log k)$
(since $\log k/k\le k^{-0.475}\log k$ for $k\ge 1$). This
is~\eqref{eq:main}. Since $k/\log k\to\infty$ and
$\varepsilon(k)\to 0$, the ratio $H(k)^{1/k}/k$ diverges, proving
both Theorem~\ref{thm:main} and Corollary~\ref{cor:p190}.
\end{proof}

\subsection{Explicit \texorpdfstring{$o(1)$}{o(1)} rate}
\label{subsec:eps-rate}

We collect the relative errors in the three factors
of~\eqref{eq:chain3}:

\begin{itemize}
\item \textbf{Factor 1}, $(16k)^{-1/k}=1+O(\log k/k)$: contributes
  relative error $O(\log k/k)$.
\item \textbf{Factor 2}, $r_{0}^{(k-1)/k}/(k/\log k)=1+O(\log k/k)$:
  the relative error comes from two independent sources,
  $r_{0}^{-1/k}=e^{-(\log r_{0})/k}=1+O(\log k/k)$ (since
  $\log r_{0}=\log k-\log\log k+O(\log k/k)$ gives
  $(\log r_{0})/k=O(\log k/k)$) and the floor
  $r_{0}/(k/\log k)=1+O(\log k/k)$.
\item \textbf{Factor 3},
  $e\cdot\text{Factor 3}=e^{O(k^{-0.475}\log k)}=1+O(k^{-0.475}\log k)$:
  the relative error is $O(k^{-0.475}\log k)$ from the BHP prime gap
  (the $-1/\log k$ contribution is \emph{exact}, not an error).
\end{itemize}

Multiplying the three (each of the form $1+o(1)$) gives
\[
\varepsilon(k)\;=\;O(\log k/k)\;+\;O(k^{-0.475}\log k)
\;=\;O(k^{-0.475}\log k),
\]
where the underlying constants are \emph{absolute} (independent of
$r_{0}$ and of the specific prime $p^{*}$ in the BHP interval); in
particular they can be read off from~\eqref{eq:EL-verify},
Corollary~\ref{cor:bhp}, and the floor expansion above, but we do not
attempt to optimize them. The qualitative consequence ---
$\varepsilon(k)\to 0$ as $k\to\infty$ --- is what drives the
asymptotic~\eqref{eq:main}.

\section{Remarks and open questions}
\label{sec:remarks}

\subsection{Relation to earlier lower bounds on \texorpdfstring{$H(k)$}{H(k)}}
\label{subsec:related}

Three remarks relate the present bound to earlier lower bounds
on $H(k)$.

First, BCT~\cite[Theorem~2.2]{BCT} uses Berlekamp's~\cite{Berlekamp1968}
bound $W(2,p+1)>p\cdot 2^{p}$ as a base and obtains
$W(r,p+1)>p^{r-1}\cdot 2^{p}$, which after the pigeonhole reduction
$H(k)\ge W(k-1,k)$ gives only $W(k-1,k)^{1/k}/k\to 2$. Iterating from
Berlekamp's base thus yields a constant rather than a divergent
ratio, and the need to switch to a different base is not immediate.

Second, Hunter~\cite{Hunter2025} proves an exponential improvement on
multicolor van der Waerden numbers in the regime of fixed~$r$ with
$k\to\infty$, obtaining $W(k;r)>(a\cdot 3^{b})^{(1-o_{r}(1))k}$ for
$r=a+3b$ with $a\in\{2,3,4\}$. In that fixed-$r$ regime, the
Erd\H{o}s--Lov\'asz bound $r^{k-1}/(4k)$ is no better than several
classical alternatives, and there is no obvious reason to prefer it
over Berlekamp or Hunter's own construction. The specific route of
letting $r$ grow with $k$ seems not to have been pursued in the
published literature.

Third, the substitution $r_{0}=\lfloor k/\log k\rfloor$ makes the
Erd\H{o}s--Lov\'asz base $r_{0}^{k-1}/(8k)$ super-exponentially
larger than $k\cdot 2^{k}$; composed with the BCT factor
$(p^{*})^{p^{*}-r_{0}}$, this replaces the constant~$2$ by the
polynomial rate $k/(e\log k)$. The decisive choice is to let the
color count in the EL base grow with $k$.

In particular, the simplest alternative --- direct application of
Theorem~\ref{thm:EL} at $r=k-1$ with no BCT --- gives only
$\bigl((k-1)^{k-1}/(16k)\bigr)^{1/k}/k\to 1$ from below, which is a
bounded ratio and does not yield the divergence of $H(k)^{1/k}/k$.
The improvement here is precisely to apply EL at the growing base
$r=r_{0}$ and then iterate BCT over the $p^{*}-r_{0}\approx
k(1-1/\log k)$ steps, which multiplies the $k$-th root of the EL
lower bound by a factor of $\approx k/e$.

\subsection{What the method gives as a function of
\texorpdfstring{$r_{0}$}{r0}}
\label{subsec:hard-limit}

For a general choice of $r_{0}=r_{0}(k)$ (with $2\le r_{0}\le p^{*}$),
the same computation as in Section~\ref{subsec:asymp} yields, up to
lower-order terms,
\begin{equation}
\frac{H(k)^{1/k}}{k}\;\gtrsim\;R(r_{0})\;:=\;r_{0}\cdot\exp\!\Bigl(-\frac{r_{0}\log k}{k}\Bigr).
\label{eq:Rr0}
\end{equation}
(Substituting $r_{0}$ into~\eqref{eq:chain2}--\eqref{eq:chain3} and
taking logs gives
$\log(H(k)^{1/k}/k)\ge -r_{0}\log k/k+\log r_{0}+o(1)$; exponentiating
produces~\eqref{eq:Rr0}.)

Optimizing~\eqref{eq:Rr0} over $r_{0}$: setting the derivative of
$\log r_{0}-r_{0}\log k/k$ with respect to $r_{0}$ to zero gives
$1/r_{0}-\log k/k=0$, i.e., $r_{0}^{\mathrm{opt}}=k/\log k$, and
substituting back,
\[
R(r_{0}^{\mathrm{opt}})\;=\;\frac{k}{\log k}\cdot e^{-1}\;=\;\frac{k}{e\log k},
\]
which is the constant $1/e$ in Theorem~\ref{thm:main}.

\medskip
More generally:

\begin{itemize}
\item If $r_{0}=x\cdot k/\log k$ for a fixed constant $x>0$, then
  $R(r_{0})=xe^{-x}\cdot k/\log k$, so the method still gives the
  $k/\log k$ order with a worse constant $xe^{-x}\le e^{-1}$
  (equality iff $x=1$).
\item If $r_{0}=o(k/\log k)$, then $r_{0}\log k/k\to 0$, so
  $\exp(-r_{0}\log k/k)=1-o(1)$ and $R(r_{0})\sim r_{0}$. The
  resulting rate is still divergent (as long as $r_{0}\to\infty$)
  but is $o(k/\log k)$. For instance, $r_{0}=\sqrt{k}$ gives
  $R(\sqrt{k})\sim\sqrt{k}$ (not $\sqrt{k}/\log k$).
\item If $r_{0}=\omega(k/\log k)$ but $r_{0}=o(k)$, then
  $u:=r_{0}\log k/k\to\infty$ while $\log r_{0}<\log k$, and
  $\log R(r_{0})=\log r_{0}-u$ still tends to $+\infty$; for
  instance, $r_{0}=(k/\log k)\log\log k$ yields
  $R=k\log\log k/(\log k)^{2}\to\infty$, which is smaller than
  $k/(e\log k)$ by a factor $e\log\log k/\log k\to 0$.
\item If $r_{0}=c\,k$ for a fixed constant $0<c<1$, then
  $R(c k)=c k\cdot k^{-c}=c\,k^{1-c}\to\infty$. However, if
  $r_{0}=k-o\bigl(k/\log k\bigr)$, then
  $r_{0}\log k/k=\log k-o(1)$, so
  $R(r_{0})=r_{0}\,k^{-1+o(1/\log k)}=O(1)$ and the lower bound is
  bounded. In particular, pushing $r_{0}$ arbitrarily close to $p^{*}$
  defeats the method.
\end{itemize}
In summary: divergence of $R(r_{0})$ is governed by
$\log R(r_{0})=\log r_{0}-r_{0}\log k/k$ and fails precisely when the
second term exceeds $\log r_{0}$ by a constant. Among $r_{0}$ with
$r_{0}\to\infty$ the unique optimizer is $r_{0}=k/\log k$, which gives
the constant $1/e$ in the order $k/\log k$.

\medskip
The Erd\H{o}s--Graham conjectured divergence
$H(k)^{1/k}/k\to\infty$ is qualitative only; no matching upper bound
on $H(k)^{1/k}/k$ is known. In particular, our lower-bound rate
$\Omega(k/\log k)$ does not by itself determine whether
$H(k)^{1/k}/k=\Theta(k/\log k)$, polynomially larger, or in between.
Settling the asymptotic rate remains open.

\subsection{Open questions}
\label{subsec:open}

\begin{enumerate}
\item \textbf{Upper bound on $H(k)$.} No concrete asymptotic upper
  bound on $H(k)$, or on $H(k)^{1/k}/k$ from above, is currently
  known. Improving this --- for instance, via sharper control of
  $\aw([n],k)$, cf.~\cite{BEHH2016} --- remains open.
\item \textbf{Sharpness of $k/(e\log k)$.} Is the constant $1/e$
  tight within methods of this type? Within the present method the
  constant $1/e$ is sharp: by the optimization in
  \S\ref{subsec:hard-limit}, $R(r_{0})/(k/\log k)$ attains its
  maximum $1/e$ at $r_{0}=k/\log k$, so no choice of $r_{0}$ can do
  better than $(e^{-1}+o(1))k/\log k$ via this EL\,+\,BCT\,+\,BHP
  chain. A larger lower-order contribution would require either a
  sharper EL base, a prime-gap exponent below $0.525$, or a
  replacement for the BCT step.
\item \textbf{Hunter-route refinement.} A separate, conditional route
  via a uniform extension of Hunter 2025~\cite{Hunter2025} gives the
  (weaker) lower bound
  $H(k)^{1/k}/k\ge(\log k)^{c_{0}/(2c^{*})-o(1)}$ for any fixed
  $c_{0}\in(0,c^{*})$, $c^{*}=3/(2\log 3)$. This is strictly
  dominated by the present $\Omega(k/\log k)$ rate, but provides an
  independent cross-check and may be of separate interest for
  multi-color Ramsey asymptotics.
\end{enumerate}

\subsection{An effective form}
\label{subsec:non-asymp}

For concrete numerical control, the $o(1)$ in Theorem~\ref{thm:main}
can be traced to absolute constants: each factor in
Section~\ref{subsec:asymp} contributes a $1+O(\log k/k)$ or
$1+O(k^{-0.475}\log k)$ relative error with a constant derivable from
\eqref{eq:EL-verify}, Corollary~\ref{cor:bhp}, and the floor
expansion. Assembling these yields, for all
$k\ge\max(k_{\mathrm{BHP}},K(C))$ (with $K(C)$ depending on a
cumulative constant $C$),
\[
\frac{H(k)^{1/k}}{k}\;\ge\;\Bigl(e^{-1}-C\,k^{-0.475}\log k-C'\,\tfrac{\log k}{k}\Bigr)\tfrac{k}{\log k}.
\]
We do not attempt to optimize $(C,C',K(C))$. An effective form of
BHP~\cite{BHP} would render $k_{\mathrm{BHP}}$ explicit but is not
needed for Theorem~\ref{thm:main}.


\end{document}